\documentclass[11pt, letterpaper]{article}

\usepackage{geometry}
\usepackage{color}
\usepackage{verbatim}
\usepackage[utf8x]{inputenc}
\usepackage{t1enc}
\usepackage[english]{babel}

\usepackage{amsmath,amssymb,amsthm}
\usepackage{mathrsfs,esint}
\usepackage{graphicx,wrapfig}
\usepackage[format=hang]{caption}

\newcommand*{\R}{{\mathbb R}}

\newcommand*{\N}{{\mathbb N}}
\newcommand*{\Z}{{\mathbb Z}}

\newcommand*{\Capp}{\text{Cap}}

\newcommand*{\eps}{\varepsilon}

\newcommand*{\Om}{\Omega}

\newcommand*{\pip}{\varphi}

\providecommand*{\vint}[1]{\mathchoice
          {\mathop{\vrule width 5pt height 3 pt depth -2.5pt
                  \kern -9pt \kern 1pt\intop}\nolimits_{\kern -5pt{#1}}}
          {\mathop{\vrule width 5pt height 3 pt depth -2.6pt
                  \kern -6pt \intop}\nolimits_{\kern -3pt{#1}}}
          {\mathop{\vrule width 5pt height 3 pt depth -2.6pt
                  \kern -6pt \intop}\nolimits_{\kern -3pt{#1}}}
          {\mathop{\vrule width 5pt height 3 pt depth -2.6pt
                  \kern -6pt \intop}\nolimits_{\kern -3pt{#1}}}}
\newcommand*{\jint}{\fint}
\DeclareMathOperator{\Lip}{Lip}

\DeclareMathOperator{\dist}{dist}

\DeclareMathOperator{\rad}{rad}

\numberwithin{equation}{section}
\theoremstyle{plain}
\newtheorem{thm}[equation]{Theorem}

\newtheorem{prop}[equation]{Proposition}

\newtheorem{lem}[equation]{Lemma}

\theoremstyle{definition}

\newtheorem{remark}[equation]{Remark}

\begin{document}

\title{Trace and extension theorems for homogeneous Sobolev and Besov spaces for unbounded uniform domains in metric measure spaces}
\author{Ryan Gibara, Nageswari Shanmugalingam}
\maketitle

\begin{center}
Dedicated to Professor O.~V.~Besov on the occasion of his 90th birthday.\\
\end{center}

\begin{abstract} {In this paper we fix $1\le p<\infty$ and consider $(\Om,d,\mu)$ be an unbounded, locally compact, non-complete 
metric measure space equipped with a doubling measure $\mu$ supporting a $p$-Poincar\'e inequality 
such that $\Om$ is a uniform domain in its completion $\overline\Om$. We 
realize the trace of functions in the Dirichlet-Sobolev space $D^{1,p}(\Om)$ on the 
boundary $\partial\Om$ as functions in the homogeneous Besov space $HB^\alpha_{p,p}(\partial\Om)$ for 
suitable $\alpha$; here, $\partial\Om$ is equipped with a non-atomic Borel regular measure $\nu$. We 
show that if $\nu$ satisfies a $\theta$-codimensional condition with respect to $\mu$ for some $0<\theta<p$, 
then there is a bounded linear trace operator $T:D^{1,p}(\Om)\rightarrow HB^{1-\theta/p}(\partial\Om)$ and
a bounded linear extension operator $E:HB^{1-\theta/p}(\partial\Om)\rightarrow D^{1,p}(\Om)$ 
that is a right-inverse of $T$.}
\end{abstract}

\noindent
    {\small \emph{Key words and phrases}: {Besov spaces, traces, Newton-Sobolev spaces, unbounded
uniform domain, doubling measure, Poincar\'e inequality}

\medskip


\section{Introduction}

In investigating the extension to which Dirichlet problems on a Euclidean domain 
can be posed in the study of partial differential equations, O.~V.~Besov~\cite{Besov, Besov2} formulated the notion of
Besov spaces, thus extending the work of Nikolski\u{\i}~\cite{N}. It was seen in~\cite{Besov, Besov2, BIKLN,BIN} and the
series of papers~\cite{Besov3, Besov4, Besov5} that for certain bounded Euclidean Lipschitz domains $\Om\subset\R^n$,
traces of Sobolev functions $W^{1,p}(\Om)$ belong to the Besov space $B^{1-1/p}_{p,p}(\partial\Om)$, 
where $\Om$ denotes the boundary of $\Om$, see also~\cite{Gag}. In his papers, Besov refers to the Besov 
spaces as \emph{Lipschitz-type spaces}.
The subsequent work of Jonsson and Wallin~\cite{JW80, JW84} extended this identification
of certain Besov spaces as trace class of Sobolev spaces for more irregular Euclidean domains, namely
domains whose boundaries are Ahlfors regular sets. 
In~\cite[Chapter~10]{Maz}, Maz'ya gives a detailed account of Besov capacities of Euclidean sets, using both the
homogeneous and inhomogeneous versions of Besov spaces.
Thus, every Besov function on the boundary of such a domain is permissible as a boundary condition
in the study of elliptic Dirichlet problems.

The exploration of traces of Sobolev function has been extended to the setting of metric measure 
spaces where the measure is doubling and supports a suitable Poincar\'e 
inequality. There, uniform domains $\Om$ whose boundary $\partial\Om$ satisfies a natural $\theta$-codimensional Hausdorff measure
condition for $0<\theta<p$ were shown in~\cite{Mal} to satisfy the condition that the traces of functions in the Newton-Sobolev space 
$N^{1,p}(\Om)$ are in the Besov class $B^{1-\theta/p}_{p,p}(\partial\Om)$. 
The preprint~\cite{Mal}, however, required the domain to be bounded. Subsequently, this result was extended for unbounded uniform
domains with \emph{bounded} boundaries in~\cite{GKS}.

In the event that the boundary of the uniform domain is unbounded, it is natural to ask what trace 
theorems hold true when $N^{1,p}(\Om)$ is replaced by its homogeneous analogue, the Dirichlet-Sobolev space 
$D^{1,p}(\Om)$, and $B^{1-\theta/p}_{p,p}(\partial\Om)$ is replaced by the homogeneous Besov space, 
$HB^{1-\theta/p}_{p,p}(\partial\Om)$. The primary goal of the present paper is to address this question, 
the answer to which is summarized in the following, the main theorem of the present note.

\begin{thm}\label{thm:main}
Let $(\Om,d,\mu)$ be an unbounded, locally compact, non-complete doubling metric measure space that supports a $p$-Poincar\'e inequality
for some $1\leq p<\infty$, and in addition $\Om$ be a uniform domain in its completion $\overline{\Om}$. Suppose also that 
$\partial\Om:=\overline{\Om}\setminus\Om$, the boundary of $\Om$, is unbounded and
supports a non-atomic Borel regular measure $\nu$ satisfying the following $\theta$-codimensional condition for some $0<\theta<p$\,:  
there exists a constant $C\ge 1$
such that for each $\zeta\in\partial\Om$ and
$r>0$, we have 
\[
\frac{1}{C}\, \nu(B(\zeta,r)\cap\partial\Om)\le \frac{\mu(B(\zeta,r)\cap\Om)}{r^\theta}\le C\, \nu(B(\zeta,r)\cap\partial\Om).
\]
Then there is a bounded linear trace operator $T: D^{1,p}(\Om)\to HB^{1-\theta/p}_{p,p}(\partial\Om)$ such that we have
\[
\lim_{r\to 0^+}\jint_{B(\zeta,r)\cap\Om}\!|u-Tu(\zeta)|\, d\mu=0
\]
for $\nu$-a.e.~$\zeta\in\partial\Om$ whenever $u\in D^{1,p}(\Om)$. Moreover, there is a bounded linear extension 
operator $E:HB^{1-\theta/p}_{p,p}(\partial\Om)\to D^{1,p}(\Om)$
such that $T\circ E$ is the canonical identity operator on $HB^{1-\theta/p}_{p,p}(\partial\Om)$.
\end{thm}

The proof of the above theorem, adapting the technique of~\cite{Mal}, will be in two parts; the trace part is proved in 
Theorem~\ref{thm:trace},
and the extension part is proved in Theorem~\ref{thm:ext}. The reader might also be interested in~\cite{HM} for a discussion of
trace classes of Haj\l asz-Sobolev functions on Euclidean domains satisfying a John-type condition.

We do not know whether choice of homogeneous versions of Besov and Sobolev spaces in the above theorem 
can be replaced with their inhomogeneous counterparts.
In~\cite{Mal}, where $\Om$ is bounded,  
the homogeneous spaces in the above theorem coincide with the corresponding inhomogeneous spaces; 
in this case, there is even control of the 
$L^p$-norms of the respective functions. When $\Om$ is an unbounded uniform domain but with
$\partial\Om$ bounded, then the identity of certain {\emph{inhomogeneous}}
Besov classes of functions on $\partial\Om$ as the trace of 
Dirichlet-Sobolev classes of functions on $\Om$ follows from~\cite{GKS}. Thus, the novelty in the present work is the ability 
to handle the possibility that $\partial\Om$ is unbounded. 

Note that when $p=1$, the theorem forces $0<\theta<1$. This is necessary as, when $\theta=1$, the trace class of $N^{1,1}(\Om)$
is known to be $L^1(\partial\Om)$. Indeed, in  the case that $\theta=1$, there is a \emph{linear} extension operator from  
$B^{0}_{1,1}(\partial\Om)$ to $N^{1,1}(\Om)$, but the trace operator from $N^{1,1}(\Om)$ is onto $L^1(\partial\Om)$,
with the extension from $L^1(\partial\Om)$ being necessarily non-linear, see~\cite{Gag, Pe} for the Euclidean setting, and~\cite{MSS} for the 
setting of metric measure spaces. In the case that $0<\theta<1$, the
extension operator we obtain \emph{is} bounded and linear. For the case of Euclidean domains, there is a nice discussion of alternate
definitions of trace given in~\cite{BuMa1, BuMa2}, and an accessible version of this can also be found in~\cite[Chapter~9.5]{Maz}.

Slight modifications throughout the paper show that the theorem still holds if we regard $\Omega$ as a domain living inside a larger 
complete metric measure space $X$ (as opposed to $\overline{\Om}$). Since the problem of traces in that setting can be reduced to the case that
the ambient space is merely $\overline{\Om}$, we leave this detail to the interested reader.

The link between Newton-Sobolev or Dirichlet-Sobolev spaces and the homogeneous or inhomogeneous 
Besov spaces give us a handy way of
analyzing the behavior of potentials related to Besov energy, a non-local energy, by utilizing the now 
well-known behavior of potentials related to
Dirichlet-Sobolev energy, see for example~\cite{CKKSS, LS}. Conversely, the identification of Besov 
spaces as traces of Sobolev-type spaces
also gives us the limit on the type of boundary data that give rise to finite-energy solutions, on the 
domain, of certain Dirichlet boundary value problems, see for
instance~\cite{BBS}. We hope the results given in this paper help further this endeavor of connecting 
non-local energies to local energies.

\vskip .3cm

\noindent {\bf Acknowledgement:} The authors thank Riikka Korte and Mathav Murugan
for valuable discussions on matters related to this paper.
 The second author's work is partially supported by the NSF (U.S.A.) grant DMS~\#2054960.

\section{Preliminaries}\label{Sec:2}

In this section, we develop the background material needed for the remainder of the paper. In what follows, 
$(Z,d_Z,\mu_Z)$ is an arbitrary metric measure space unless otherwise stated.

\subsection{Sobolev spaces}
	
	In a metric measure space with no {\it a priori} smooth structure, let alone linear structure, there is no one natural candidate for the 
notion of derivative. One possibility, which generalizes the fundamental theorem of calculus and exploits the geometry of curves in a 
metric measure space, is the notion of upper gradients, first proposed by
Heinonen and Koskela~\cite{HK}. 

We say that a non-negative Borel function $g$ on $Z$ is an \emph{upper gradient} of a measurable function $u$ on $Z$ if 
\begin{equation}\label{ug}
	|u(y)-u(x)|\le \int_\gamma\! g\, ds
\end{equation}
holds for all rectifiable cures in $Z$ joining $x$ to $y$. The right-hand side is meant to be interpreted as infinity if at 
least one of $u(x)$ or $u(y)$ is infinite. Every function trivially has $g\equiv\infty$ as an upper gradient, and for each upper 
gradient $g$ the function $g+\tilde{g}$ is also an upper gradient for every non-negative Borel function $\tilde{g}$. For $1\le p<\infty$, 
we say that $g$ is a {\emph{$p$-weak upper gradient}} of $u$ if the collection $\Gamma$ of rectifiable curves for 
which inequality~\eqref{ug} fails has $p$-modulus zero. Here, by a family $\Gamma$ of curves having $p$-modulus zero we
mean that there is a non-negative Borel function $\rho\in L^p(Z)$ such that $\int_\gamma\rho\, ds=\infty$ for each $\gamma\in\Gamma$.
	
	Of special importance in the context of metric measure spaces are the Lipschitz functions, which, in some sense, 
play a similar role to that of the smooth functions in real analysis. If $u$ is $L$--Lipschitz, then it is immediate that $g\equiv L$ is an 
upper gradient for $u$. For a merely locally Lipschitz function $u$, then its \emph{local Lipschitz constant function}
\[
\Lip u(x)=\limsup_{y\rightarrow x}\frac{|u(y)-u(x)|}{d_Z(y,x)}
\]
is an upper gradient for $u$.  
	
	For a fixed measurable function $u$ on $Z$, consider the collection $D_p(u)$ of all $p$-weak upper gradients of $u$. The set 
$D_p(u)\cap L^p(Z)$ is a closed convex subset of $L^p(Z)$ and so, if it is non-empty, has a unique element of smallest $L^p$-norm. 
We denote this element by $g_u$ and call it the {\emph{minimal $p$-weak upper gradient}} of $u$. 
	
	We say that a measurable function $u$ on $Z$ is in the {\emph{Dirichlet-Sobolev space}} $D^{1,p}(Z)$ for $1\leq p<\infty$ if the following 
semi-norm is finite: $\|u\|_{D^{1,p}}:=\|g_u\|_{L^p}$. If, in addition, $u$ satisfies $\int_Z|u|^p\, d\mu_Z<\infty$, then $u$ is said to be 
in the {\emph{Newton-Sobolev space}} $N^{1,p}(Z)$ with 
semi-norm $\|u\|_{N^{1,p}}:=\|u\|_{L^p}+\|u\|_{D^{1,p}}$.  

In the context of Euclidean domains, $N^{1,p}(Z)$ corresponds to the classical Sobolev spaces $W^{1,p}(Z)$. We invite
the interested reader to consult~\cite{HKST} for details and proofs regarding the statements made in this subsection.

\subsection{Sobolev $p$-capacities}\label{sub:capacities}

Let $1\leq p <\infty$. Given a set $E\subset Z$, we set its {\emph{Sobolev $p$-capacity}} to be the number
\[
\Capp_p^Z(E):=\inf_{u\in\mathcal{A}(E)} \Vert u\Vert_{N^{1,p}}\,
\]
where $\mathcal{A}(E)$ is the collection of all functions $u\in N^{1,p}(Z)$ such that $u\ge 1$ on $E$.

A function in $L^p(Z)$ is well-defined only up to sets of $\mu_Z$-measure zero. Newton-Sobolev functions are more constrained, for they are
well-defined up to sets of Sobolev $p$-capacity zero in the sense that if $u\in N^{1,p}(Z)$, then $\Vert u\Vert_{N^{1,p}}=0$ if and only if
the $p$-capacity of the set $\{z\in Z\, :\, u(z)\ne 0\}$ is zero, see for instance~\cite[Corollary~7.2.10]{HKST}.

\subsection{Doubling property of measure}

The metric measure space $(Z,d_Z,\mu_Z)$ is said to be {\emph{doubling}} if there is a constant $C\geq 1$ 
such that for all $z\in Z$ and $r>0$ we have
\[
\mu_Z(B(z,2r))\le C\, \mu_Z(B(z,r))\,.
\]
Given a ball $B\subset Z$, there may be more than one choice of center $z$ and radius $r$; we will assume that a generic ball $B$
is identified together with its center and radius. For a ball $B=B(z,r)$, we will denote by $\tau B$ the  ball $B(z,\tau r)$ when $\tau>0$.

The doubling property of $\mu_Z$ implies that $(Z,d_Z)$ is a {\emph{doubling metric space}}. That is, there exists a positive integer $N$, 
depending only on the doubling constant of $\mu_Z$, such that for each $r>0$ and $x\in Z$, every $r/2$--separated set 
$A\subset B(x,r)$ has at most $N$ elements. A set being $r/2$--separated means that for each $y,z\in A$ with $y\ne z$, 
we have $d_Z(y,z)\ge r/2$.

\subsection{Poincar\'e inequalities}

Let $1\le p<\infty$. The metric measure space $(Z,d_Z,\mu_Z)$ is said to support a {\emph{$p$-Poincar\'e inequality}} if 
there are constants $C>0$ and $\lambda\ge 1$
such that for all balls $B\subset Z$ and upper gradients $g$ of function $u\in D^{1,p}(Z)$,
\[
\jint_B\!|u-u_B|\, d\mu_Z\le C\, \rad(B)\, \left(\jint_{\lambda B}\!g^p\, d\mu_Z\right)^{1/p}.
\]
Support of a Poincar\'e inequality implies some strong geometric connectivity properties of $Z$; see~\cite{HKST} and the 
references therein for
more on this topic. The validity of $p$-Poincar\'e inequality automatically implies that functions in $D^{1,p}(Z)$ 
are necessarily in $L^1_{loc}(Z)$.

When $(Z,d_Z,\mu_Z)$ is doubling and supports a $p$-Poincar\'e inequality, a stronger version of the Lebesgue 
differentiation theorem is known for 
Newton-Sobolev functions. Recall that if $\mu_Z$ is doubling, then $\mu_Z$-almost every point in $Z$ is a 
Lebesgue point of a function $u\in L^p(Z)$.
From~\cite[Theorem~9.2.8]{HKST} for the case $p>1$ and from~\cite{KKST} for the case $p=1$ 
(see also~\cite{KV} for a related Sobolev function-space called the Haj\l asz-Sobolev space), we have the
following result. Note that when $u\in D^{1,p}(Z)$, for each ball $B\subset Z$ we have that $u\,\eta_B$ 
is in the Newton-Sobolev class $N^{1,p}(Z)$ where
$\eta_B$ is a Lipschitz function on $Z$ with support in $2B$ such that $\eta_B=1$ on $B$ and $0\le \eta_B\le 1$.

\begin{prop}\label{prop:Lebesgue-pt}
If $(Z,d_Z,\mu_Z)$ is complete,
doubling, and supports a $p$-Poincar\'e inequality, then for each $u\in D^{1,p}(Z)$ the set of non-Lebesgue points of $u$ is of
Sobolev $p$-capacity zero.
\end{prop}

The homogeneous space $D^{1,p}(Z)$ is, in some instances, different from $N^{1,p}(Z)$. Note that 
$N^{1,p}(Z)+\R\subset D^{1,p}(Z)$ in the sense
that adding constants to functions in $N^{1,p}(Z)$ gives a function that is in $D^{1,p}(Z)$. However, $D^{1,p}(Z)$ could be larger than
$N^{1,p}(Z)+\R$, see for example~\cite[Theorem~1.4, Proposition~7.3, Example~7.1]{BBS1}. Currently, to the best our knowledge, no potential
non-trivial criteria are known that characterize when $D^{1,p}(Z)=N^{1,p}(Z)+\R$. A similar question for the homogeneous and inhomogeneous
Besov spaces $HB^\alpha_{p,p}(Z)$, $B^\alpha_{p,p}(Z)$ can be posed;
these spaces are defined in the next subsection below. The above-mentioned relationships 
between the homogeneous and inhomogeneous spaces
have implications to existence problems related to global energy minimizers and potential theory.

\subsection{Besov spaces}

The study of a specific 
sub-class of Besov spaces was first initiated, in the context of $Z$ being a smooth Euclidean space, by
Nikolski\u{\i}~\cite{N} in relation to ``fractional derivatives" of functions in a generalized Zygmund class. These were then extended to more 
general Besov classes $B^\alpha_{p,q}(Z)$ by O.~V.~Besov~\cite{Besov, Besov2}. Motivated by Dirichlet problems
for Lipschitz domains in Euclidean spaces, the papers~\cite{Besov2, Besov4} developed the theory of Besov spaces as traces, 
to the boundary of the domain, of 
certain Sobolev function classes on the domain. In~\cite{Besov2}, one can also find the identification of Besov spaces as interpolation spaces,
interpolated between $L^p$ and Sobolev spaces, in the context of Euclidean spaces, see also~\cite{Kar, Tri,Y} for some 
discussion on this aspect of the theory. From the point of view of
interpolation in the context of metric measure spaces,
Besov spaces were first studied in~\cite{GKS2}. The context of traces in the metric setting, under various limitations on the 
shape of the domain in the metric space,
can be found in~\cite{GKS,Mal, MSS} for instance. The aim of the present note is to extend this aspect of traces to the 
case where both the domain and its boundary are
unbounded.
	
	We say that a function $f\in L^p_{\text{loc}}(Z)$ is in the {\emph{homogeneous Besov space}} $HB^\alpha_{p,q}(Z)$ 
for $0\leq\alpha<\infty$, $1\leq p<\infty$, and $1\leq q\leq\infty$ if the following semi-norm is finite:
\[
\|f\|_{HB^\alpha_{p,q}}:=
\begin{cases} \left(\displaystyle \int_{0}^{\infty}\left( \displaystyle \int_{Z}\displaystyle \jint_{B(y,r)}\!\frac{|f(y)-f(x)|^p}{r^{\alpha p}}\,d\mu_Z(x)\,d\mu_Z(y) \right)^{\frac qp}\frac{dr}{r} \right)^{\frac 1q}\!,&q<\infty\\
	\sup\limits_{r>0}\left(\displaystyle \int_{Z}\displaystyle \jint_{B(y,r)}\!\frac{|f(y)-f(x)|^p}{r^{\alpha p}}\,d\mu_Z(x)\,d\mu_Z(y)\right)^{\frac 1p}\!,&q=\infty
\end{cases}. 
\]
	If, in addition, $f\in L^p(Z)$, then $f$ is said to be in the {\emph{inhomogeneous Besov space}} ${B}^\alpha_{p,q}(Z)$ 
with semi-norm $\|f\|_{{B}^\alpha_{p,q}}=\|f\|_{L^p}+\|f\|_{HB^\alpha_{p,q}}$. Note that
the case $q=\infty$ is related to the so-called Korevaar-Schoen spaces, see for instance~\cite{KMc} or~\cite[Section~4]{ABCRST}.
In the present paper, we focus on the classes $B^\alpha_{p,p}(Z)$ for suitable choice of $\alpha$, as these spaces
arise in the theory of traces of Sobolev functions on $Z$. Such Besov spaces enjoy the following characterization, the proof of 
which is included for the benefit of the interested reader, see also \cite{BP,GKS2}.

\begin{lem}
Assume that $\mu_Z$ is doubling and has no atoms. For $\alpha>0$, $1\leq p<\infty$, and $f\in L^p_{\text{loc}}(Z)$,
\begin{equation}\label{eq:B-P}
	\|f\|_{HB^\alpha_{p,p}}^p
	\approx \int_{Z}\int_{Z}\!\frac{|f(y)-f(x)|^p}{d_Z(y,x)^{\alpha p}\mu_Z(B(y,d_Z(y,x))) }\,d\mu_Z(x)\,d\mu_Z(y)\,,
\end{equation}
and, for each $C>0$, we have
\begin{equation}\label{rem:partition}
	\|f\|_{HB^\alpha_{p,p}}^p
	\approx 
	\sum_{l\in\Z}\frac{1}{2^{l\alpha p}}\int_{Z}\jint_{B(y,C2^l)}\!|f(y)-f(x)|^p\,d\mu_Z(x)\,d\mu_Z(y)\,. 
\end{equation}
\end{lem}

\begin{proof}
Let $f\in L^p_{\text{loc}}(Z)$. Fix $y\in Z$ and partition $(0,\infty)$ into intervals of the form $(C2^{l-1},C2^l)$ for some $C>0$ and $l\in\Z$. 
For $C2^{l-1}<r<C2^l$, we have that
\begin{equation}\label{eq:step1}
    \jint_{B(y,r)}\!\frac{|f(y)-f(x)|^p}{r^{\alpha p + 1}}\,d\mu_Z(x)\approx \frac{1}{2^{l(\alpha p +1)}}\jint_{B(y,C2^l)}\!|f(y)-f(x)|^p\,d\mu_Z(x)\,,
\end{equation}
and so, since $\mu_Z$ is non-atomic, \eqref{rem:partition} follows. 

Setting $A_i=B(y,C2^i)\setminus B(y,C2^{i-1})$, \eqref{eq:step1} also tells us, using the doubling property of $\mu_Z$, that
\[
\jint_{B(y,r)}\!\frac{|f(y)-f(x)|^p}{r^{\alpha p + 1}}\,d\mu_Z(x)\approx
   \frac{1}{2^{l(\alpha p +1)}\mu_Z(B(y,2^l))}\sum_{i=-\infty}^{l}\int_{A_i}\!|f(y)-f(x)|^p\,d\mu_Z(x)\,.
\]
Hence,
\begin{align*}
	\int_{0}^{\infty}\!\jint_{B(y,r)}\!\frac{|f(y)-f(x)|^p}{r^{\alpha p + 1}}\,d\mu_Z(x)
	&\approx \sum_{l\in\Z}\frac{1}{2^{l\alpha p}\mu_Z(B(y,2^l))}\sum_{i=-\infty}^{l}\int_{A_i}\!|f(y)-f(x)|^p\,d\mu_Z(x)\\
	& = \sum_{i\in\Z}\left(\sum_{l=i}^{\infty} \frac{1}{2^{l\alpha p}\mu_Z(B(y,2^l))}\right) \int_{A_i}\!|f(y)-f(x)|^p\,d\mu_Z(x). 
\end{align*}
Since 
\begin{align*}
\frac{1}{2^{l\alpha p}\mu_Z(B(y,2^i))}\leq\sum_{l=i}^{\infty} \frac{1}{2^{l\alpha p}\mu_Z(B(y,2^l))}
&\leq \frac{1}{2^{i\alpha p}\mu_Z(B(y,2^i))}\sum_{l=i}^{\infty}\frac{1}{2^{(l-i)\alpha p}}\\
&\lesssim \frac{1}{2^{i\alpha p}\mu_Z(B(y,2^i))}\,,
\end{align*}
we have that 
\begin{align*}
\int_{0}^{\infty}\!\jint_{B(y,r)}\!\frac{|f(y)-f(x)|^p}{r^{\alpha p + 1}}\,d\mu_Z(x)& \approx 
\sum_{i\in\Z}\frac{1}{2^{i\alpha p}\mu_Z(B(y,2^i))}\int_{A_i}\!|f(y)-f(x)|^p\,d\mu_Z(x)\\
&\approx \sum_{i\in\Z}\frac{1}{2^{i\alpha p}\mu_Z(B(y,2^i))}\int_{A_i}\!|f(y)-f(x)|^p\,d\mu_Z(x).
\end{align*}
For $x\in A_i$, we have that $d_Z(y,x)\approx 2^i$ and so $\mu_Z(B(y,d_Z(y,x)))\approx \mu_Z(B(y,2^i))$ by doubling and 
monotonicity of measure; therefore,
\begin{align*}
\int_{0}^{\infty}\!\jint_{B(y,r)}\!\frac{|f(y)-f(x)|^p}{r^{\alpha p + 1}}\,d\mu_Z(x)
&\approx \sum_{i\in\Z}\int_{A_i}\!\frac{|f(y)-f(x)|^p}{d_Z(y,x)^{\alpha p}\mu_Z(B(y,d_Z(y,x)))}\,d\mu_Z(x)\\
&= \int_{Z}\!\frac{|f(y)-f(x)|^p}{d_Z(y,x)^{\alpha p}\mu_Z(B(y,d_Z(y,x)))}\,d\mu_Z(x).
\end{align*}
An application of Fubini's theorem then yields~\eqref{eq:B-P}. 
\end{proof}

\subsection{Uniform domains}

Uniform domains were first introduced by Martio and Sarvas~\cite{MaSa} in the context of quasiconformal mappings between 
Euclidean domains, and
since then, they have been used extensively in different contexts, including quaisconformal mapping theory, potential theory, and PDEs.

If $(Z,d_Z)$ is a complete metric space and $\Om$ is a locally compact, non-complete domain in $Z$, its boundary 
is the set $\partial\Om:=\overline{\Om}\setminus \Om$,
where $\overline{\Om}$ is the metric completion of the non-complete space $\Om$ with respect to the metric $d_Z$. 
For $z\in Z$, we set 
\[
d_\Om(z):=\dist(z,\partial \Om):=\min\{d_Z(z,x)\, :\, x\in\partial\Om\}.
\] 
Since $\Om$ is locally compact,
it follows that $\Om$ is open in $\overline{\Om}$ and so $d_\Om(z)>0$ when $z\in\Om$.

The domain $\Om$ 
 is said to be a \emph{uniform domain} if there is a constant $A\ge 1$ such that
whenever $x,y\in\Om$, we can find a curve $\gamma$ in $\Om$ with end points $x,y$ such that 
\begin{enumerate}
\item[(i)] the length $\ell(\gamma)\le A\, d_Z(x,y)$,
\item[(ii)] for each point $z$ in the trajectory of $\gamma$ we have 
\[
 \min\{\ell(\gamma_{x,z}),\ell(\gamma_{z,y})\}\le A\,d_\Om(z),
\]
where $\gamma_{x,z}$ denotes any segment of $\gamma$ with end points $x,z$, and a similar interpretation for $\gamma_{z,y}$ holds.
\end{enumerate}

From~\cite{BS} we know that if $\Om$ is a uniform domain in a metric measure space $(Z,d_Z,\mu_Z)$ such that the metric measure space is 
doubling and supports a $p$-Poincar\'e inequality, then the restriction of the measure $\mu_Z$ and the metric $d_Z$ to $\Om$ also yields 
a metric measure space that is doubling and supports a $p$-Poincar\'e inequality.

\subsection{Standing assumptions}

Let $(\Om,d,\mu)$ be an unbounded, locally compact metric measure space such that $\Om$ is uniform in its completion 
$\overline\Om$. We assume $(\Om,d,\mu)$ is doubling 
and satisfies a $p$-Poincar\'{e} inequality, $1\leq p<\infty$, and that there exists a 
non-atomic
Borel regular measure $\nu$ on $\partial\Om$ that is 
$\theta$--codimensional to $\mu$, $0<\theta<p$, in the sense that there exists  $C\geq 1$ for which
\begin{equation}\label{eq:codim}
C^{-1}\frac{\mu(B(\zeta,r)\cap\Om)}{r^\theta}\leq\nu(B(\zeta,r)\cap\partial\Om)\leq C\frac{\mu(B(\zeta,r)\cap\Om)}{r^\theta}
\end{equation}
holds for all $\zeta\in\partial\Om$ and $r>0$. Note that $\mu$ being doubling on $\Om$ implies that $\nu$ is doubling on $\partial\Om$. 
	
	We will often consider $\Om$ as a domain living within the metric measure space $(\overline\Om,d,\mu)$, where 
$\mu$ is extended by zero to $\partial\Om$. Since $\mu$ is doubling on $\Om$, its extension is doubling 
on $\overline\Om$. Moreover, since $\Om$ supports a $p$-Poincar\'e inequality, then so does 
$\overline{\Om}$, see~\cite[Proposition~7.1]{AS}.
Hence we have that $D^{1,p}(\Om)=D^{1,p}(\overline{\Om})$ and $N^{1,p}(\Om)=N^{1,p}(\overline{\Om})$. 

As discussed in Subsection~\ref{sub:capacities}, under the above standing assumptions, it follows from 
Proposition~\ref{prop:Lebesgue-pt} that for $u\in D^{1,p}(\overline{\Om})$ 
the complement of the Lebesgue points of $u$ has Sobolev $p$-capacity zero. Hence, by~\cite[Proposition~3.11, Lemma~8.1]{GKS}
we know that $\nu$-almost every point in $\partial\Om$ is a Lebesgue point of $u$.
For greater details on the above, we refer the reader to Subsection~\ref{Sec3.2} below.

\section{On traces}\label{Sec:3}

In this section, we consider the trace of Dirichlet-Sobolev functions under the aforementioned standing assumptions.  
We can assume also without loss
of generality that $\Om$ is $A$-uniform domain with $A\ge 2$.

\subsection{Constructing cones}\label{subsec:cones}

Here we fix a parameter $\tau\ge 1$ (in our application, we will choose $\tau=\lambda$ where $\lambda$ is the scaling factor
on the right-hand side of the Poincar\'e inequality). The following construction is based on~\cite[Lemma~4.3]{BS}. 
Let $\xi,\zeta\in\partial\Om$, and let $\gamma$ be a uniform curve in $\Om$ with end points $\zeta,\xi$, that is,
$\gamma:[0,\ell(\gamma)]\to\overline{\Om}$ such that $\gamma(0)=\xi$, $\gamma(\ell(\gamma))=\zeta$,
and $\gamma((0,\ell(\gamma)))\subset\Om$ is a uniform curve. 
Let $x_0=\gamma(\ell(\gamma)/2)$, the mid-point of the curve $\gamma$.
We focus on the subcurve $\gamma_{\xi,x_0}$ of $\gamma$ to construct the balls $B_k$ for $k\ge 0$, with the
similar construction for $\gamma_{x_0,\zeta}$ giving balls $B_k$ for $k<0$.

Let $r_0=\tfrac{d_\Om(x_0)}{16\tau}$, and set $B_0=B(x_0,r_0)$. Next, let $x_1$ be the point in the trajectory of
$\gamma_{\xi,x_0}$ to be the 
last point at which $\gamma_{\xi,x_0}$ leaves the ball $B_0$ so that $\gamma_{\xi,x_1}$
does not intersect $B_0$. If $d_\Om(x_1)\ge \tfrac{d_\Om(x_0)}{2}$, then we choose $r_1=r_0$; if
$d_\Om(x_1)<\tfrac{d_\Om(x_0)}{2}$, then we set $r_1=\tfrac{d_\Om(x_1)}{16\tau}$. We then set 
$B_1=B(x_1,r_1)$. Continuing inductively, once $x_k$ and $r_k$ have been selected for some positive integer $k$,
let $x_{k+1}$ be the last point in $\gamma_{\xi,x_0}$ at which $\gamma_{\xi,x_0}$ leaves 
$\bigcup_{j=0}^k B_k$, namely, $\gamma_{\xi,x_{k+1}}$ does not intersect $\bigcup_{j=0}^k B_k$,
but $\gamma_{x_{k+1},x_k}\setminus\{x_{k+1}\}\subset \bigcup_{j=0}^kB_k$. We then set
$r_{k+1}=r_k$ if $d_\Om(x_{k+1})\ge 8\tau r_k$, and $r_{k+1}=\tfrac{d_\Om(x_{k+1})}{16\tau}$ otherwise.
Note that in either case,
\begin{equation}\label{eq:rad-dOm}
d_\Om(x_{k+1})\ge 8\tau r_{k+1}.
\end{equation}

Now we consider the properties of the chain of balls $B_k$, $k=0, 1,\ldots.$ Clearly, $B_k\cap B_{k+1}$ is not empty.
We fix a non-negative integer $j$ such that $r_j=\tfrac{d_\Om(x_j)}{16\tau}$,
and let $k>j$ be such that $r_k=r_j$. Let $L=\ell(\gamma_{\xi,x_j})$, and
$l=\ell(\gamma_{x_k,x_j})$. As $r_k=r_j$, we have that $d_\Om(x_k)\ge 8\tau r_j=d_\Om(x_j)/2$. It then follows from
the $A$-uniformity of $\gamma$  that
\[
\frac{d_\Om(x_j)}{2}\le d_\Om(x_k)\le d(x_k,\xi)\le L-l\le Ad_\Om(x_j)-l.
\]
Thus, $l\le (A-\tfrac12)d_\Om(x_j)$, and so
\begin{equation}\label{eq:length-control}
k-j\le \frac{l}{r_j}\le (A-\tfrac12)\, \frac{d_\Om(x_j)}{d_\Om(x_j)/(16\tau)}\le 16\tau A.
\end{equation}
If $k=j+1$ such that $r_k\ne r_j$, then $r_k<r_j/2$, and so we have that $\lim_kr_k=0$. As each $x_k$ lies on the
curve $\gamma_{\xi,x_0}$, it follows that $\lim_kx_k=\xi$. 

Next, fix a non-negative integer $j$ for which $r_j=\tfrac{d_\Om(x_j)}{16\tau}$, 
and let $k>j$ be the smallest integer for which $r_k\ne r_j$. Then we know that $r_k=\tfrac{d_\Om(x_k)}{8\tau}$ and
$d_\Om(x_k)<\tfrac{d_\Om(x_j)}{2}$, with $d_\Om(x_{k-1})\ge \tfrac{d_\Om(x_j)}{2}$. 
It follows from triangle
inequality that
\[
d_\Om(x_k)\ge d_\Om(x_{k-1})-d(x_k,x_{k-1})\ge \frac{d_\Om(x_j)}{2}-r_j
=\left[1-\frac{1}{8\tau}\right] \, \frac{d_\Om(x_j)}{2},
\]
and so
\begin{equation}\label{eq:rad-control}
\left[1-\frac{1}{8\tau}\right] \, \frac{d_\Om(x_j)}{2}\le d_\Om(x_k)<\frac{d_\Om(x_j)}{2}.
\end{equation}
From~\eqref{eq:length-control} and~\eqref{eq:rad-control} we see that there is some constant $K>1$ (which depends
only on $A$ and $\tau$) such that for each $k\ge 0$ we have 
\begin{equation}\label{eq:rk-vs-dOm}
8\tau r_k\le d_\Om(x_k)\le K\, r_k.
\end{equation}
By the $A$-uniformity of $\gamma$ we also have that 
\[
d_\Om(x_k)\le d(\xi, x_k)\le A\, d_\Om(x_k).
\]
It follows that for $x\in 4\tau B_k$, we have that 
\[
d_\Om(x)\approx d_\Om(x_k)\approx d(x,\xi)\approx d(x_k,\xi).
\]
Now suppose that $k$ and $j$ are non-negative integers with $k>j$ such that $4\tau B_k\cap 4\tau B_j$ is non-empty
and that $r_j\ne r_k$. Then by~\eqref{eq:rad-dOm}, we have
\begin{equation}\label{eq:gen-x}
d_\Om(x_j)-d_\Om(x_k)\le d(x_j,x_k)\le 4\tau(r_j+r_k)\le \frac{d_\Om(x_j)+d_\Om(x_k)}{2},
\end{equation}
from which we obtain $d_\Om(x_j)\le 3\, d_\Om(x_k)$. Note that for positive integers $m,n$ with $m\ne n$, it is possible
to have $r_n=r_m$. As pointed out in the discussion above, if $r_m\ne r_{m-1}=r_n=\tfrac{d_\Om(x_n)}{8\tau}$, then 
$d_{\Om}(x_m)<\tfrac{d_\Om(x_n)}{2}$. If in the string of positive integers between $j$ and $k$ we had $N$ distinct
values for $r_m$, $j\le m\le k$, then by~\eqref{eq:rad-control}, necessarily we have 
$d_\Om(x_k)<\tfrac{d_\Om(x_j)}{2^{N-1}}\le \frac{3\, d_\Om(x_k)}{2^{N-1}}$, and so we must have $2^{N-1}<3$,
that is, $N\le 2$. It follows now from~\eqref{eq:length-control} that 
$2^{-j}\le 2^{-k+N_0}$ for some positive integer $N_0$ that depends solely on $A$ and $\tau$, 
that is, $k\le j+N_0$. Thus we have shown that if $k,j$ are non-negative integers so that
$4\tau B_j\cap 4\tau B_k$ is non-empty, then $|k-j|\le N_0$. Combining this with~\eqref{eq:length-control} we see that
there is a constant $C\ge 1$ such that 
\begin{equation}\label{eq:bdd-overlap}
\sum_{k=0}^\infty \chi_{4\tau B_k}\le C,
\end{equation}
that is, we have a bounded overlap of the enlarged balls $4\tau B_k$.

The cones $C[\xi,\zeta]$ and $C[\zeta,\xi]$ are the sets
\begin{equation}\label{eq:Cones}
C[\xi,\zeta]:=\bigcup_{k=0}^\infty 4\tau B_k,\qquad C[\zeta,\xi]:=\bigcup_{k=0}^\infty 4\tau B_{-k}.
\end{equation}
Note also that $d_\Om(x_0)\approx d(\zeta,\xi)$ with the comparison constant depending solely on $A$.

\subsection{The co-dimensional measure on $\partial\Om$ and the existence of traces}\label{Sec3.2}

Recall that we assume $\Om$ to support a $p$-Poincar\'e inequality. It follows that 
given a function $u\in D^{1,p}(\Om)$, and a compactly supported Lipschitz function $\eta$ on $\Om$, the function
$\eta u$ is in the inhomogeneous Sobolev class $N^{1,p}(\Om)$. It follows from~\cite[Proposition~7.1]{AS} that 
$\eta u$ has an extension to $\partial\Om$ such that the extended function lies in $N^{1,p}(\overline{\Om})$.
As the minimal $p$-weak upper gradient of a function is determined by the local behavior of the function, it follows that
$u$ itself has an extension to $\partial\Om$ such that the extended function lies in $D^{1,p}(\overline{\Om})$;
that is, $D^{1,p}(\Om)=D^{1,p}(\overline{\Om})$.

 From~\cite[Proposition~3.11]{GKS} (with $U=\overline{\Om}$ and $\mu_Z$ the measure $\mu$,
and with $A=E\subset\partial\Om\subset U$)
we know that whenever $E\subset\partial\Om$ is a set such that
the Sobolev capacity $\text{Cap}_p^{\overline{\Om}}(E)=0$, then necessarily the codimensional Hausdorff measure of $E$,
$\mathcal{H}^{-t}(E)$, is zero for $0<t<p$. From~\cite[Lemma~8.1]{GKS} we know that when $0<\theta<p$, necessarily
$\mathcal{H}^{-\theta}\vert_{\partial\Om}\approx \nu$, and so it follows that
$\nu(E)=0$.

Having established that $u\in D^{1,p}(\Om)$ has an extension to a function in $D^{1,p}(\overline{\Om})$ and that
(by the Poincar\'e inequality on $\overline{\Om}$) $p$-capacity almost every point in $\partial\Om$ is a Lebesgue point
of $u$, and that $p$-capacity zero subsets of $\partial\Om$ are $\nu$-null,
see~\cite[Theorem~9.2.8]{HKST},  
we have that $\nu$-a.e. point in $\partial\Om$ is a Lebesgue point of 
$u\in D^{1,p}(\overline{\Om})$. Thus for each $u\in D^{1,p}(\Om)$ there is a set $E_u\subset\partial\Om$ with $\nu(E_u)=0$
such that whenever $\zeta\in \partial\Om\setminus E_u$, there is a real number, denoted $Tu(\zeta)$, such that
\[
\lim_{r\to 0^+}\jint_{B(\zeta,r)\cap\Om}|u-Tu(\zeta)|\, d\mu=0.
\]

\subsection{The trace theorem}

In this subsection we finally identify the trace relationship, the first part of Theorem~\ref{thm:main}.

\begin{thm}\label{thm:trace}
Let $1\leq p <\infty$ and $0<\theta<p$. Then there is a bounded linear trace operator 
$T:D^{1,p}(\Om)\to HB^{1-\theta/p}_{p,p}(\partial\Om)$ such that when $u\in D^{1,p}(\Om)$, we have
\[
 Tu(\zeta)=\lim_{r\to 0^+} \jint_{B(\zeta,r)}u\, d\mu
\]
for $\nu$-almost every $\zeta\in\partial\Om$. 
\end{thm}

In the above setting, we can also consider $\nu$ to be a measure on $\overline{\Om}$ by extending $\nu$ by zero to $\Om$;
a similar null-extension of $\mu$ to $\partial\Om$ would allow us to simplify notation (by not needing to use $B(\zeta,r)\cap\partial\Om$
but merely using $B(\zeta,r)$ for instance in talking about the measure $\nu$ of the balls).

\begin{proof}[Proof of Theorem~\ref{thm:trace}]
For $\zeta,\xi\in\partial\Om$ that are $\mu$-Lebesgue points of $u$, we 
use the chain of balls $B_k$, $k\in\Z$ from Subsection~\ref{subsec:cones} above, with the choice of $\tau=\lambda$,
where $\lambda$ is the scaling constant associated with the Poincar\'e inequality. 

Let $u\in D^{1,p}(\Om)=D^{1,p}(\overline{\Om})$; the discussion of Subsection~\ref{Sec3.2} tells us that 
$\nu$-almost every point $\zeta\in\partial\Om$ is a $\mu$-Lebesgue point of such $u$, and hence $Tu(\zeta)$ is well-defined.
For the rest of the proof, we will continue to denote $Tu(\zeta)$ by $u(\zeta)$ as this does not give rise to conflict of notation.
Then, fixing $\eps>0$ such that $\theta+\eps<p$, 
\begin{align*}
|u(\zeta)-u(\xi)|\le \sum_{k\in\Z} |u_{B_k}-u_{B_{k+1}}|
  &\lesssim \sum_{k\in\Z} r_k \left(\jint_{4\lambda B_k}g_u^p\, d\mu\right)^{1/p} \\
  &= \sum_{k\in\Z} r_k^{1-(\theta+\eps)/p}\, r^{(\theta+\eps)/p}\left(\jint_{4\lambda B_k}g_u^p\, d\mu\right)^{1/p}\\
  &\le \left(\sum_{k\in\Z} r_k^{\theta+\eps}\jint_{4\lambda B_k}g_u^p\, d\mu\right)^{1/p}
      \left(\sum_{k\in\Z} r_k^{\tfrac{p-\theta-\eps}{p-1}}\right)^{1-1/p}.
\end{align*}
Note that
\[
\sum_{k\in\Z} r_k^{\tfrac{p-\theta-\eps}{p-1}}\approx d(\xi,\zeta)^{\tfrac{p-\theta-\eps}{p-1}}
\, \sum_{k\in\Z} 2^{-k\, \tfrac{p-\theta-\eps}{p-1}}.
\]
Since the sum on the right-hand side of the above expression is finite (and independent of $\xi,\zeta$), it follows that
\[
\left(\sum_{k\in\Z} r_k^{\tfrac{p-\theta-\eps}{p-1}}\right)^{1-1/p}\approx d(\zeta,\xi)^{1-\tfrac{\theta+\eps}{p}}.
\]
Hence,
\[
|u(\zeta)-u(\xi)|^p\lesssim d(\xi,\zeta)^{p-\theta-\eps}\ \sum_{k\in\Z} r_k^{\theta+\eps}\jint_{4\lambda B_k}g_u^p\, d\mu
  \approx d(\xi,\zeta)^{p-\theta-\eps}\ \sum_{k\in\Z}\frac{r_k^{\theta+\eps}}{\mu(B_k)}\int_{4\lambda B_k}g_u^p\, d\mu.
\]
By the codimensionality condition on $\partial\Om$, and by the doubling property of $\mu$, we have
\[
\mu(B_k)\approx\mu(4AB_k)\approx \mu(B(\omega_k, r_k)\cap\Om)\approx r_k^\theta\, \nu(B(\omega_k,r_k)\cap\partial\Om),
\]
where $\omega_k=\zeta$ for $k>0$ and $\omega_k=\xi$ for $k\le 0$. It follows that
\[
\frac{|u(\zeta)-u(\xi)|^p}{d(\xi,\zeta)^{p-\theta}}\lesssim d(\xi,\zeta)^{-\eps} \
  \sum_{k\in\Z}\frac{r_k^{\eps}}{\nu(B(\omega_k,r_k)\cap\partial\Om)}\int_{4\lambda B_k}g_u^p\, d\mu.
\]

Setting 
\[
I:=\sum_{k=0}^\infty \frac{r_k^{\eps}}{\nu(B(\zeta,r_k)\cap\partial\Om)}\int_{4\lambda B_k}g_u^p\, d\mu
\]
and
\[
II:=\sum_{k=1}^{-\infty}\frac{r_k^{\eps}}{\nu(B(\xi,r_k)\cap\partial\Om)}\int_{4\lambda B_k}g_u^p\, d\mu,
\]
we have the following:
\begin{align*}
I&\approx \sum_{k=1}^\infty  \int_{4\lambda B_k}\frac{d(x,\zeta)^\eps g_u(x)^p}{\nu(B(\zeta,d(\zeta,x))\cap\partial\Om)}\, d\mu(x)
 =\int_{C[\zeta,\xi]}\frac{d(x,\zeta)^\eps g_u(x)^p}{\nu(B(\zeta,d(\zeta,x))\cap\partial\Om)}\, d\mu(x),\\
II&\approx \sum_{k=0}^{-\infty}\int_{4\lambda B_k}\frac{d(x,\xi)^\eps g_u(x)^p}{\nu(B(\xi,d(\xi,x))\cap\partial\Om)}\, d\mu(x)
 =\int_{C[\xi,\zeta]}\frac{d(x,\xi)^\eps g_u(x)^p}{\nu(B(\xi,d(\xi,x))\cap\partial\Om)}\, d\mu(x),
\end{align*}
where we have used the fact that $r_k\approx d(x,\omega_k)\approx d_\Om(x_k)$ for each $x\in 4\lambda B_k$,
see~\eqref{eq:gen-x} together with~\eqref{eq:rk-vs-dOm}.
We have denoted in the above $C[\zeta,\xi]=\bigcup_{k=0}^\infty 4\lambda B_k$ and 
$C[\xi,\zeta]=\bigcup_{k=1}^{-\infty} 4\lambda B_k$, as in Subsection~\ref{subsec:cones}, and used the fact that the balls 
$4\lambda B_k$ are of bounded
overlap, see~\eqref{eq:bdd-overlap}. 

Now, we write
\[
\frac{|u(\zeta)-u(\xi)|^p}{d(\xi,\zeta)^{p-\theta}}
\lesssim 
\int_{C[\zeta,\xi]}\frac{d(\xi,\zeta)^{-\eps}d(x,\zeta)^\eps g_u(x)^p}{\nu(B(\zeta,d(\zeta,x))\cap\partial\Om)}\, d\mu(x)
+\int_{C[\xi,\zeta]}\frac{d(\xi,\zeta)^{-\eps}d(x,\xi)^\eps g_u(x)^p}{\nu(B(\xi,d(\xi,x))\cap\partial\Om)}\, d\mu(x).
\]
Hence,
\[
\int_{\partial\Om}\int_{\partial\Om}\frac{|u(\zeta)-u(\xi)|^p}{d(\xi,\zeta)^{p-\theta}\nu(B(\zeta,d(\zeta,\xi))\cap\partial\Om)}\, d\nu(\zeta)\, d\nu(\xi)
  \lesssim E+F,
\]
where
\[
E:=\int_{\partial\Om}\int_{\partial\Om}\int_{C[\zeta,\xi]}\frac{d(\xi,\zeta)^{-\eps}d(x,\zeta)^\eps g_u(x)^p}{\nu(B(\zeta,d(\zeta,x))\cap\partial\Om)\, \nu(B(\zeta,d(\zeta,\xi))\cap\partial\Om)}\, d\mu(x)\, d\nu(\zeta)\, d\nu(\xi)
\]
and
\[
F:=\int_{\partial\Om}\int_{\partial\Om}\int_{C[\xi,\zeta]}\frac{d(\xi,\zeta)^{-\eps}d(x,\xi)^\eps g_u(x)^p}{\nu(B(\xi,d(\xi,x))\cap\partial\Om)\, \nu(B(\zeta,d(\zeta,\xi))\cap\partial\Om)}\, d\mu(x)\, d\nu(\zeta)\, d\nu(\xi).
\]
We estimate $E$ as follows. We first note that for $x\in\Om$, if $x\in C[\zeta,\xi]$ then necessarily 
$d(\xi,\zeta)\ge d_\Om(x)/A$ by the uniformity of the curve that was used to generate $C[\zeta,\xi]$, and, moreover,
$Ad_\Om(x)\ge d(\zeta,x)$. Therefore,
\begin{align*}
E&=\int_{\partial\Om}\int_{\partial\Om}\int_{\Om}\frac{d(\xi,\zeta)^{-\eps}d(x,\zeta)^\eps g_u(x)^p\, \chi_{C[\zeta,\xi]}(x)}{\nu(B(\zeta,d(\zeta,x))\cap\partial\Om)\, \nu(B(\zeta,d(\zeta,\xi))\cap\partial\Om)}\, d\mu(x)\, d\nu(\zeta)\, d\nu(\xi)\\
&\le \int_{\partial\Om}\int_{\partial\Om}\int_{\Om} \frac{d(\xi,\zeta)^{-\eps}d(x,\zeta)^\eps g_u(x)^p\, \chi_{B(x,Ad_\Om(x))}(\zeta)\, \chi_{\partial\Om\setminus B(\zeta,d_\Om(x)/A)}(\xi)}{\nu(B(\zeta,d(\zeta,x))\cap\partial\Om)\, \nu(B(\zeta,d(\zeta,\xi))\cap\partial\Om)}\, d\mu(x)\, d\nu(\zeta)\, d\nu(\xi)\\
&=\int_\Om g_u(x)^p\, \int_{\partial\Om}\int_{\partial\Om\setminus B(\zeta,d_\Om(x)/A)}\frac{d(\xi,\zeta)^{-\eps}d(x,\zeta)^\eps\, \chi_{B(x,Ad_\Om(x))}(\zeta)}{\nu(B(\zeta,d(\zeta,x))\cap\partial\Om)\, \nu(B(\zeta,d(\zeta,\xi))\cap\partial\Om)}\, d\nu(\xi)\, d\nu(\zeta)\, d\mu(x).
\end{align*}
In the last line above we used Tonelli's theorem.
To estimate the inner-most integral, for each positive integer $j$ we set 
\[
A_j=\left(B(\zeta, 2^jd_\Om(x)/A)\setminus B(\zeta, 2^{j-1}d_\Om(x)/A)\right)\cap\partial\Om,
\] 
and see from the doubling property of $\nu$ that
\begin{align*}
\int\limits_{\partial\Om\setminus B(\zeta,d_\Om(x)/A)}\frac{d(\xi,\zeta)^{-\eps}}{\nu(B(\zeta,d(\zeta,\xi))\cap\partial\Om)}\, d\nu(\xi)
&=\sum_{j=1}^\infty \int_{A_j}\frac{d(\xi,\zeta)^{-\eps}}{\nu(B(\zeta,d(\zeta,\xi))\cap\partial\Om)}\, d\nu(\xi)\\
&\approx \sum_{j=1}^\infty (2^jd_\Om(x)/A)^{-\eps}\\ & \approx d_\Om(x)^{-\eps}.
\end{align*}
Hence, as $d(x,\zeta)\approx d_\Om(x)$ when 
$\zeta\in B(x,Ad_\Om(x))\cap\partial\Om\subset \left(B(x,Ad_\Om(x))\setminus B(x,d_\Om(x))\right)\cap\partial\Om$, it follows that 
\begin{align*}
E&\lesssim 
\int_\Om g_u(x)^p\, \int_{B(x,Ad_\Om(x))\cap\partial\Om}\frac{d(x,\zeta)^\eps\, d_\Om(x)^{-\eps}}{\nu(B(\zeta,d(\zeta,x))\cap\partial\Om)}\, d\nu(\zeta)\, d\mu(x)\\
&\approx \int_\Om g_u(x)^p\, \int_{B(x,Ad_\Om(x))\cap\partial\Om}\frac{1}{\nu(B(\zeta,d(\zeta,x))\cap\partial\Om)}\, d\nu(\zeta)\, d\mu(x).
\end{align*}
Again, by the above observation about $d(\zeta,x)$ for $\zeta\in B(x,Ad_\Om(x))\cap\partial\Om$, it follows that
for each such $\zeta$ we have $\nu(B(\zeta,d(\zeta,x))\cap\partial\Om)\approx \nu(B(x,Ad_\Om(x))\cap\partial\Om)$ 
via the doubling property of $\nu$. Hence
\[
E\lesssim \int_\Om g_u(x)^p\, d\mu(x).
\]
A similar treatment of the term $F$ yields
\[
F\lesssim \int_\Om g_u(x)^p\, d\mu(x).
\]

In conclusion, we obtain the desired estimate
\[
\int_{\partial\Om}\int_{\partial\Om}\frac{|u(\zeta)-u(\xi)|^p}{d(\xi,\zeta)^{p-\theta}\nu(B(\zeta,d(\zeta,\xi))\cap\partial\Om)}\, d\nu(\zeta)\, d\nu(\xi)
  \lesssim \int_\Om g_u(x)^p\, d\mu(x).
\]
\end{proof}

\section{On extensions}\label{Sec:4}

In this section, we consider the extension of Besov functions from the boundary to the entire domain. 
To construct the extension operator, we consider a partition of unity subordinate to a Whitney cover.  

\subsection{Whitney coverings and constructing the extension operator}

	Since $\partial\Om$ is non-empty and $(\Om,d)$ is a doubling metric space, we are able to construct a Whitney 
decomposition of $\Om$; that is, a countable collection $W_\Om$ of balls $B_{i,j}=B(x_{i,j},r_{i,j})$, $i\in\Z$ and 
$j\in\N$, in $\Om$ satisfying the following:
 \begin{enumerate}
 \item[(i)] $\Om=\bigcup\limits_{i,j}B_{i,j}$;
 \item[(ii)] there exists a constant $C>0$ such that $\sum\limits_{i,j}\chi_{2B_{i,j}}\leq C$;
 \item[(iii)] for each $i\in\Z$, $2^{i-1}<r_{i,j}\leq 2^i$ for all $j\in\N$;
 \medskip
 \item[(iv)] and, for each $i\in\Z$ and $j\in\N$, $r_{i,j}=\frac{1}{8}\,d_\Om(x_{i,j})$.
 \end{enumerate}
The elements of $W_\Om$ are called Whitney balls. See for example~\cite[Proposition~4.1.15]{HKST}; a 
simple modification of the proof found
there yields our desired Whitney decomposition. 

\begin{remark}\label{rem:bnded-overlap}
	The constant $C$ in (ii) above depends only on $N$ from the definition of a doubling metric space, 
which in turn depends only on the doubling constant of $\mu$. In fact, we can choose this cover in such 
a way that for each $\sigma\ge 1$, there is a constant $N_\sigma$ which depends only on $\sigma$ and the 
doubling constant of $\mu$, such that for each
$i\in\Z$ and $j\in\N$ there are at most $N_\sigma$ many indices $k\in\N$ for which $\sigma B_{i,j}\cap\sigma B_{i,k}$ is non-empty.
\end{remark}
	
Note also that by construction, for $x\in 2B_{i,j}$ the triangle inequality gives that 
\[
d_\Om(x)\geq d_\Om(x_{i,j}) - d(x,x_{i,j})>8r_{i,j}-2r_{i,j}>0, 
\]
and so $2B_{i,j}\subset\Om$ for each $i\in\Z$ and $j\in\N$.
	
\begin{lem}\label{claim:intersect}
If $2B_{i,j}\cap B_{l,m}\neq\emptyset$, then $|i-l|\leq 3$.
\end{lem}

\begin{proof}
We begin by assuming that $i>l$. Then, by the triangle inequality, and properties (iii) and (iv), we have that 
\[
d(x_{l,m},x_{i,j})\geq d_\Om(x_{i,j})-d_\Om(x_{l,m}) = 8r_{i,j}-8 r_{l,m} > 2^{i+2}-2^{l+3}.
\]
This implies, still using property (iii), that 
\[
\dist(B_{l,m},2B_{i,j})\geq d(x_{l,m},x_{i,j})-2r_{i,j}-r_{l,m}>2^{i+2}-2^{l+3}-2^{i+1}-2^l >2^{i+1}-2^{l+4},
\]
which is positive if $i-l>3$. 
		
On the other hand, when $i<l$, similar calculations yield
\[
d(x_{i,j},x_{l,m})\geq d_\Om(x_{l,m}) - d_\Om(x_{i,j}) = 8 r_{l,m} - 8 r_{i,j}> 2^{l+2}-2^{i+3}
\]
and so 
\[
\dist(2B_{i,j},B_{l,m})\geq d(x_{i,j},x_{l,m}) - 2r_{i,j} - r_{l,m} > 2^{l+2} - 2^{i+3} - 2^{i+1} - 2^l  > 2^{l+1}-2^{i+4},
\]
which is positive if $l-i>3$.
\end{proof}
	
We now form a partition of unity subordinate to the Whitney decomposition $W_\Om$. Select 
functions $\pip_{i,j}$ satisfying the following: 
\begin{enumerate}
\item[(i')] $\chi_{\Om}=\sum\limits_{i,j}\pip_{i,j}$;
\item[(i'')] for each $i\in\Z$ and $j\in\N$, $0\leq\pip_{i,j}\leq\chi_{2B_{i,j}}$;
\medskip
\item[(iii')] and, for each $i\in\Z$ and $j\in\N$, $\pip_{i,j}$ is $C/r_{i,j}$--Lipschitz.
\end{enumerate}
	
Given a center $x_{i,j}\in\Om$ of the Whitney ball $B_{i,j}$, denote by $\hat{x}_{i,j}$ a closest point in $\partial\Om$; 
there may be more
than one such choice, but we fix one choice for each $B_{i,j}$. Then set 
$U_{i,j}:=B(\hat{x}_{i,j},r_{i,j})\cap\partial\Om$ and $U_{i,j}^*:=B(\hat{x}_{i,j}, 2^8r_{i,j})$. The number $2^8$ in the construction
of $U_{i,j}^*$ looks strange at this point of the discourse, but it is forced upon us in the proof in the next subsection, see
for instance~\eqref{radii} and~\eqref{osc}. But at this juncture, the reader will not go astray by replacing $2^8$ with any
large constant in visualizing $U_{i,j}^*$ and in the following lemma.

\begin{lem}\label{bdd-overlap-U}
There is a positive integer $N$ that depends only on the doubling constant of $\mu$ such that
for each fixed $i\in\Z$ and for each $j\in\N$,  there are at most $N$ number of sets $U_{i,k}^*$, $k\in\N$,
for which $U_{i,j}^*\cap U_{i,k}^*$ is non-empty. 
\end{lem}

\begin{proof}
Indeed, if $U_{i,j}^*$ intersects $U_{i,k}^*$, then we have
\[
d(x_{i,j},x_{i,k})\le d_\Om(x_{i,j})+d_\Om(x_{i,k})+d(\hat{x}_{i,j},\hat{x}_{i,k})=8(r_{i,j}+r_{i,k})+2^8(r_{i,j}+r_{i,k})
\le 2^9(r_{i,j}+r_{i,k}).
\]
As $r_{i,k}\le 2^i\le 2r_{i,j}$, it follows that $d(x_{i,j},x_{i,k})\le 2^{10}\, r_{i,j}$, and hence $2^{11}B_{i,j}\cap 2^{11}B_{i,k}$ is non-empty.
By the construction of the Whitney cover, it follows that there are at most $N=N_{2^{11}}$ number of such positive 
integers $k$ -- see Remark~\ref{rem:bnded-overlap}.
\end{proof}

We are finally able to construct the extension operator. Beginning with a function $f\in L^1_{\text{loc}}(\partial\Om)$, we 
construct an extension $F$ on $\Om$ by writing
\begin{equation}\label{construction}
	F(x)=\sum_{i,j}f_{U_{i,j}}\pip_{i,j}(x),
\end{equation}
where $U_{i,j}:=B(\hat{x}_{i,j},r_{i,j})\cap\partial\Om$ and $f_{U_{i,j}}:=\jint_{U_{i,j}}\!f\,d\nu$.
	
\subsection{The extension theorem}

	First, we show that the extension $F$ of a function in $f\in{HB}^{1-\theta/p}_{p,p}(\partial\Om)$ will be in $D^{1,p}(\Om)$. 

\begin{prop}\label{LipF}
If $f\in HB^{1-\theta/p}_{p,p}(\partial\Om)$, $1\leq p <\infty$, and $0<\theta<p$, then 
$F$ is locally Lipschitz continuous in $\Om$, and 
$\|\Lip F\|_{L^p}\lesssim\|f\|_{HB^{1-\theta/p}_{p,p}(\partial\Om)}$. In particular, $F\in D^{1,p}(\Om)$ with 
\[
\|F\|_{D^{1,p}}\lesssim\|f\|_{HB^{1-\theta/p}_{p,p}(\partial\Om)}.
\]
\end{prop}

\begin{proof}
Fix a Whitney ball $B_{l,m}\in W_\Om$. For any two points $x,y\in B_{l,m}$, we have from property (i') that 
\[
\sum_{i,j}(\pip_{i,j}(y)-\pip_{i,j}(x))=0,
\]
and so 
\[
|F(y)-F(x)|=\left| \sum_{i,j}f_{U_{i,j}}(\pip_{i,j}(y)-\pip_{i,j}(x)) \right|=\left| \sum_{i,j}(f_{U_{i,j}}-f_{U_{l,m}})(\pip_{i,j}(y)-\pip_{i,j}(x)) \right|.
\]
Denoting by $I(l,m)$ the collection of all $(i,j)$ such that $2B_{i,j}\cap B_{l,m}\neq\emptyset$, we have from properties (ii) and (iii') that
\begin{equation}\label{difference}
|F(y)-F(x)|\leq \sum_{I(l,m)}|f_{U_{i,j}}-f_{U_{l,m}}|\frac{d(y,x)}{r_{i,j}}.
\end{equation}
		
From Lemma~\ref{claim:intersect} we know that if $(i,j)\in I(l,m)$, then $|i-l|\leq 3$. From this and (iii) it follows that
\begin{equation}\label{radii}
\min_{I(l,m)}r_{i,j}\approx r_{l,m}\approx 2^{l}. 
\end{equation}
Lemma~\ref{claim:intersect} also implies that for $(i,j)\in I(l,m)$, 
$U_{i,j}\subset B(\hat{x}_{l,m},2^8 r_{l,m})\cap\partial\Om=:U^*_{l,m}$ and, from the doubling property of $\nu$ we 
then also have that $\nu(U^*_{l,m})\approx \nu(U_{l,m})\approx \nu(U_{i,j})$. Thus,
\begin{equation}\label{osc}
|f_{U_{i,j}}-f_{U_{l,m}}|\leq \jint_{U_{i,j}}\jint_{U_{l,m}}\!|f(\zeta)-f(\xi)|\,d\nu(\xi)d\nu(\zeta)
\lesssim \jint_{U^*_{l,m}}\jint_{U^*_{l,m}}\!|f(\zeta)-f(\xi)|\,d\nu(\xi)d\nu(\zeta). 
\end{equation}
Hence, for $x\in B_{l,m}$, property (ii) along with equations \eqref{difference}, \eqref{radii}, and \eqref{osc} imply that
\[
\Lip F(x)\lesssim \frac{1}{2^l}\jint_{U^*_{l,m}}\jint_{U^*_{l,m}}\!|f(\zeta)-f(\xi)|\,d\nu(\xi)d\nu(\zeta).
\]
Applying the doubling property of $\mu$ along with the codimensionality condition on $\nu$, we have that 
$\mu(B_{l,m})\approx\mu(U_{l,m})\lesssim r_{l,m}^\theta\nu(U_{l,m})\approx 2^{l\theta}\nu(U^*_{l,m})$, and so
\begin{align*}
\int_{\Om}\!(\Lip F)^p\,d\mu\leq\sum_{l,m}\int_{B_{l,m}}\!(\Lip F)^p\,d\mu
&\lesssim \sum_{l,m}\frac{\mu(B_{l,m})}{(2^l)^p}\left(\jint_{U^*_{l,m}}\jint_{U^*_{l,m}}\!|f(\zeta)-f(\xi)|\,d\nu(\xi)d\nu(\zeta)\right)^p\\
&\lesssim \sum_{l,m}\frac{\nu(U^*_{l,m})}{(2^l)^{p-\theta}}\jint_{U^*_{l,m}}\jint_{U^*_{l,m}}\!|f(\zeta)-f(\xi)|^p\,d\nu(\xi)d\nu(\zeta)\\
&= \sum_{l,m}\frac{1}{(2^l)^{p-\theta}}\int_{U^*_{l,m}}\jint_{U^*_{l,m}}\!|f(\zeta)-f(\xi)|^p\,d\nu(\xi)d\nu(\zeta).
\end{align*}
For $\zeta\in U^*_{l,m}$, we have 
$U^*_{l,m}\subset B(\zeta,2^9r_{l,m})\cap\partial\Om\subset B(\zeta,C\,2^{l})\cap\partial\Om$ with 
$\nu(U^*_{l,m})\approx\nu(B(\zeta,C\,2^{l})\cap\partial\Om)$ from the doubling property of $\nu$. This implies that
\begin{align*}
\int_{\Om}\!(\Lip F)^p\,d\mu&\lesssim \sum_{l,m}\frac{1}{(2^l)^{p-\theta}}\int_{U^*_{l,m}}\jint_{B(\zeta,C\,2^{l})\cap\partial\Om}\!|f(\zeta)-f(\xi)|^p\,d\nu(\xi)d\nu(\zeta)\\
&\lesssim  \sum_{l}\frac{1}{(2^l)^{p-\theta}}\int_{\partial\Om}\jint_{B(\zeta,C\,2^{l})\cap\partial\Om}\!|f(\zeta)-f(\xi)|^p\,d\nu(\xi)d\nu(\zeta)
\end{align*}
using the fact that $\{U^*_{l,m}\}_m$ has bounded overlap for each fixed $l$, see Lemma~\ref{bdd-overlap-U}. 
Finally, by~\eqref{rem:partition},
\[
\int_{\Om}\!(\Lip F)^p\,d\mu\lesssim \int_{0}^{\infty}\frac{1}{r^{p-\theta}}\int_{\partial\Om}\jint_{B(\zeta,r)\cap\partial\Om}\!|f(\zeta)-f(\xi)|^p\,d\nu(\xi)d\nu(\zeta)\frac{dr}{r}\approx\|f\|^p_{{HB}^{1-\theta/p}_{p,p}(\partial\Om)}.
\]
\end{proof}

Next, we show that this extension is really an extension. 

\begin{prop}\label{thm:L^p-trace}
If $f\in L^p_{\text{loc}}(\partial\Om)$, $1\leq p <\infty$, then
\[
\lim_{r\rightarrow 0^+}\jint_{B(\zeta,r)\cap\Om}\!|F-f(\zeta)|^p\,d\mu=0
\]
for $\nu$-a.e. $\zeta\in\partial\Om$. That is, the trace of $F$ exists and equals $f$ $\nu$-a.e.. 
\end{prop}

We will prove the above proposition using Lebesgue's differentiation theorem for the function $f$ with respect to
$\nu$, together with the following lemmas. 
	
\begin{lem}\label{prop:Lploc}
If $f\in L^p_{\text{loc}}(\partial\Om)$, then
\begin{equation}\label{eq:Lploc}
\int_{B_{l,m}}\!|F|^p\,d\mu\lesssim 2^{l\theta}\int_{U^*_{l,m}}\!|f|^p\,d\nu.
\end{equation}
\end{lem}
	
\begin{proof}
Fix a Whitney ball $B_{l,m}\in W_\Om$. Lemma~\ref{claim:intersect} implies that 
$U_{i,j}\subset U^*_{l,m}$ for all $(i,j)\in I(l,m)$  
and that $\nu(U^*_{l,m})\approx\nu(U_{i,j})$. Thus, by property (i') and by H\"older's inequality,
\begin{align*}
\int_{B_{l,m}}\!|F|^p\,d\mu=\int_{B_{l,m}}\!\left|\sum_{I(l,m)}\left(\jint_{U_{i,j}}\!f\,d\nu\right)\pip_{i,j}\right|^pd\mu
&\leq \int_{B_{l,m}}\!\left(\left(\jint_{U^*_{l,m}}\!|f|\,d\nu\right)\sum_{i,j}\pip_{i,j}\right)^pd\mu\\
&\leq\mu(B_{l,m})\jint_{U^*_{l,m}}\!|f|^p\,d\nu.
\end{align*}
As $\mu(B_{l,m})\lesssim 2^{l\theta}\nu(U^*_{l,m})$ (this follows from the doubling property of $\mu$ and the 
$\theta$-codimensionality of $\nu$), we have the desired inequality.
\end{proof}
	
\begin{lem}\label{prop:Lploc-layer}
If $f\in L^p_{\text{loc}}(\partial\Om)$ and $r>0$, then 
\[		
\int_{B(\zeta,r)\cap\Om}\!|F|^p\,d\mu\lesssim r^\theta \int_{B(\zeta,2^8r)\cap\partial\Om}\!|f|^p\,d\nu. 
\]		
\end{lem}
	
\begin{proof}
Fix $\zeta\in\partial\Om$ and $r>0$. 
Suppose that some Whitney ball $B_{i,j}$ intersects  $B(\zeta,r)$. Then by the construction of the Whitney cover, we have 
$8r_{i,j}=d_\Om(x_{i,j})\le d(\zeta,x_{i,j})$, and so it follows that $r_{i,j}\le r/7$. Let $I(r)$ be the collection of all integers $i$ for which
there is some positive integer $j$ with $B_{i,j}\cap B(\zeta,r)$ non-empty. For each $i\in I(r)$,
denote  by $\mathcal{J}(i)$ the collection of $j\in\N$ such that $B_{i,j}\cap B(\zeta,r)\neq\emptyset$. 
Then~\eqref{eq:Lploc} implies that
\[
\int_{B(\zeta,r)\cap\Om}\!|F|^p\,d\mu
\leq \sum_{i\in I(r)}\sum_{j\in\mathcal{J}(i)}\int_{B_{i,j}}\!|F|^p\,d\mu
\lesssim \sum_{i\in I(r)}\sum_{j\in\mathcal{J}(i)}2^{i\theta}\int_{U^*_{i,j}}\!|f|^p\,d\nu.
\]
For $i\in I(r)$ and $j\in\mathcal{J}(i)$, we have that $B_{i,j}\cap B(\zeta,r)\neq\emptyset$ and 
so $d(\zeta,x_{i,j})\leq r+r_{i,j}$. Recall that for each $i\in I(r)$ and $j\in\mathcal{J}(i)$ we have 
\begin{equation}\label{eq:rij-r}
2^{i-1}\le r_{i,j}\le r/7.
\end{equation}  
The triangle inequality then implies that 
$U^*_{i,j}\subset B(\zeta, r+2^9r_{i,j})\subset B(\zeta, 2^8r)$. 
By the bounded overlap property of $\{U^*_{i,j}\}_{j\in\mathcal{J}(i)}$ for each fixed $i\in I(r)$, see 
Lemma~\ref{bdd-overlap-U}, we have that
\begin{align*}
\sum_{i\in I(r)}2^{i\theta}\sum_{j\in\mathcal{J}(i)}\int_{U^*_{i,j}}\!|f|^p\,d\nu
&\lesssim \sum_{i\in I(r)}2^{i\theta}\int_{B(\zeta, 2^8r)\cap\partial\Om}|f|^p\, d\nu\\
&=2^{i_0\theta}\left(\sum_{i=0}^\infty 2^{-i\theta}\right)\, \int_{B(\zeta, 2^8r)\cap\partial\Om}|f|^p\, d\nu\\
&\approx 2^{i_0\theta}\, \int_{B(\zeta, 2^8r)\cap\partial\Om}|f|^p\, d\nu,
\end{align*}
where $i_0=\max I(r)$. Since for $i\in I(r)$ we have $2^{i-1}\le r/7$, 
see~\eqref{eq:rij-r} above,
it follows that $2^{i_0\theta}\lesssim r^\theta$; the
claim of the lemma now follows.
\end{proof}

\begin{proof}[Proof of Proposition~\ref{thm:L^p-trace}]
Fix $\zeta\in\partial\Om$ and write $f_{\zeta}(\xi)=f(\xi)-f(\zeta)$ for $\xi\in\partial\Om$. 
We have that $f_{\zeta}\in L^p_{\text{loc}}(\partial\Om)$ and its extension $F_{\zeta}$ 
satisfies $F_{\zeta}(x)=F(x)-f(\zeta)$ for every $x\in\Om$. An application of 
Lemma~\ref{prop:Lploc-layer} to $f_{\zeta}$ yields
\begin{align*}
\int_{B(\zeta,r)\cap\Om}\!|F(x)-f(\zeta)|^p\,d\mu(x)
&=\int_{B(\zeta,r)\cap\Om}\!|F_{\zeta}(x)|^p\,d\mu(x)\\
&\lesssim r^\theta \int_{B(\zeta,2^8r)\cap\partial\Om}\!|f_\zeta(\xi)|^p\,d\nu(\xi)
=r^\theta \int_{B(\zeta,2^8r)\cap\partial\Om}\!|f(\xi)-f(\zeta)|^p\,d\nu(\xi).
\end{align*}
From the doubling and codimensionality of $\nu$ it follows that 
\[
\jint_{B(\zeta,r)\cap\Om}\!|F(x)-f(\zeta)|^p\,d\mu(x)\lesssim \jint_{B(\zeta,r)\cap\partial\Om}\!|f(\xi)-f(\zeta)|^p\,d\nu(\xi).
\]
By the local $p$-integrability of $f$, $\nu$-almost every $\zeta\in\partial\Om$ is a Lebesgue point of $f$, 
and so the right-hand side of the above inequality tends to 0 as $r\rightarrow{0^+}$ for $\nu$-almost every $\zeta\in\partial\Om$. 
	\end{proof}
	
Now we are ready to prove the second part of Theorem~\ref{thm:main}.

\begin{thm}\label{thm:ext}
Let $1\leq p <\infty$ and $0<\theta<p$. There is a bounded linear extension operator 
$E:HB^{1-\theta/p}_{p,p}(\partial\Om)\to D^{1,p}(\Om)$ such that
$T\circ E$ is the identity map on $HB^{1-\theta/p}_{p,p}(\partial\Om)$, where $T$ is the trace operator constructed in
the proof of Theorem~\ref{thm:trace}.
\end{thm}

\begin{proof}
For $f\in HB^{1-\theta/p}_{p,p}(\partial\Om)$, take $Ef=F$, where $F$ is as in \eqref{construction}. Then $E$ is linear 
by construction and is bounded from $HB^{1-\theta/p}_{p,p}(\partial\Om)$ to $D^{1,p}(\Om)$ by Proposition~\ref{LipF}. 
Consider the trace operator $T$ from Theorem~\ref{thm:trace}. Then $T\circ E f = TF = f$ $\nu$-almost everywhere 
by Proposition~\ref{thm:L^p-trace}. 
\end{proof}

This completes the proof of the main theorem of this note, Theorem~\ref{thm:main}.

\noindent {\bf Address:} \\

\noindent Department of Mathematical Sciences, P.O.~Box 210025, University of Cincinnati, Cincinnati, OH~45221-0025, U.S.A.\\
\noindent E-mail: R.G.: {\tt ryan.gibara@gmail.com}, \ \ \ N.S.: {\tt shanmun@uc.edu}\\

\end{document}